\documentclass[a4paper]{amsart}
\usepackage{amssymb,amscd,verbatim,latexsym}
\usepackage{enumerate}
\usepackage[all,pdf]{xy}

\usepackage[T1]{fontenc}
\usepackage[american]{babel}
\usepackage[utf8]{inputenc}
\usepackage{lmodern}
\usepackage{microtype}

\usepackage[mathscr]{eucal}

\usepackage{color}

\newtheorem{thm}{Theorem}[section]
\newtheorem{proposition}[thm]{Proposition}
\newtheorem{corollary}[thm]{Corollary}

\newtheorem{lem}[thm]{Lemma}
\newtheorem{nota}[thm]{Notation}
\theoremstyle{definition}

\newtheorem{dfn}[thm]{Definition}
\theoremstyle{remark}
\newtheorem{rem}[thm]{Remark}
\newtheorem{remark}[thm]{Remark}
\newtheorem{example}[thm]{Example}
\newtheorem{exam}[thm]{Example}

\newcommand{\bbC}{\mathbb C}

\newcommand{\bbR}{\mathbb R}
\newcommand{\bbS}{\mathbb S}
\newcommand{\bbZ}{\mathbb Z}

\newcommand{\bj}{\mathbf j}
\newcommand{\bs}{\mathbf s}
\newcommand{\bt}{\mathbf t}

\newcommand{\bu}{\mathbf u}
\newcommand{\bm}{\mathbf m}
\newcommand{\bi}{\mathbf i}

\newcommand{\cD}{\mathcal D}
\newcommand{\cG}{\mathcal G}
\newcommand{\cH}{\mathcal H}
\newcommand{\cT}{\mathcal T}

\newcommand{\cJ}{\mathcal J}
\newcommand{\cK}{\mathcal K}

\newcommand{\fU}{\mathfrak U}

\newcommand{\rA}{A}

\newcommand{\rD}{D}
\newcommand{\rE}{E}
\newcommand{\rG}{G}

\newcommand{\rM}{M}

\newcommand{\rS}{S}
\newcommand{\rU}{U}
\newcommand{\tto}{\rightrightarrows}

\DeclareMathOperator{\Forall}{\forall}
\DeclareMathOperator{\Ker}{Ker}

\DeclareMathOperator{\ind}{ind}

\DeclareMathOperator{\id}{id}

\DeclareMathOperator{\ch}{ch}

\begin{document}

\author[Yu Qiao]{Yu Qiao}

\address{School of Mathematics and Statistics, Shaanxi Normal University,
Xi'an, 710119, Shaanxi, China} \email{yqiao@snnu.edu.cn}

\author[Bing Kwan So]{Bing Kwan So}
\address{School of Mathematics, Jilin University,
Changchun, 130023, Jilin, China} \email{bkso@graduate.hku.hk}


%\thanks{Qiao was partially supported by the NSFC(11971282).}

\date{\today}

\title[Deformation and $K$-theoretic Index Formulae]{Deformation and $K$-theoretic Index Formulae on Boundary Groupoids}

\begin{abstract}
Motivated by investigating $K$-theoretic index formulae for boundary groupoids of the form 
$$\cH = M_0 \times M_0 \sqcup G \times M_1 \times M_1 \tto M = M_0 \sqcup M_1,$$ 
where $G$ is an exponential Lie group, we introduce the notion of a 
{\em deformation from the pair groupoid}, which makes sense for general Lie groupoids. 
Once there exists a deformation from the pair groupoid for a (general) Lie groupoid $\cG \tto M$, 
we are able to construct explicitly a deformation index map relating the analytic index on $\cG$ and the index on the pair groupoid $M \times M$, 
%and prove an `index comparison theorem' between these two indices,
which in turn enables us to establish index formulae for (fully) elliptic (pseudo)-differential operators on $\cH$ by applying the numerical index formula of 
M. J. Pflaum, H. Posthuma, and X. Tang.
In particular, we find that the index is given by the Atiyah-Singer integral but does not involve any eta term in the higher codimensional cases.
These results recover and generalize our previous results for renormalizable boundary groupoids via renormalized traces.
\end{abstract}

\maketitle

\tableofcontents

\section{Introduction}

The Atiyah-Singer index theorem is one of the greatest mathematics achievements in the twentieth century,
which states that the analytic index of an elliptic differential operator is equal to
its topological counterpart.

There have been many results generalizing the Atiyah-Singer index theorem to
other pseudo-differential calculi, constructed for different purposes
\cite{Rouse;BlupInd,Nistor;LieMfld,Bohlen;Rev1,Debord;FiberedCorners,Debord;Blup,M'bert;CornerGroupoids,NWX;GroupoidPdO}.
These pseudo-differential calculi can be realized as groupoid pseudo-differential calculi on certain groupoid $\cG \rightrightarrows \rM$ with $\rM$ compact.
Then the sub-algebras of operators of order zero and $-\infty$ can be completed to $C^*$-algebras
$\fU (\cG)$ and $C^* (\cG)$, respectively, which yields the short exact sequence
\begin{equation}
\label{Exact}
0 \to C^* (\cG) \to \fU (\cG) \to \fU / C^* (\cG) \to 0.
\end{equation}
After passing to the six terms exact sequence, one defines the analytic index of any elliptic operator
to be its image in $K _0 (C^* (\cG))$ under the boundary map.
For continuous family groupoids, the analytic index is similarly defined by Lauter, Monthubert, and Nistor in \cite{Nistor;Family}.

In the 1990s, Connes gives a ``Lie groupoid proof'' of the Atiyah-Singer index theorem in his
famous book \cite[Section 2]{Connes;Book}.
As the starting point, given a closed manifold $M$,
Connes constructed the so-called {\em tangent groupoid}, which is a kind of {\em deformation} of the pair groupoid $M \times M$.
This construction induces a map (which we shall denote by $\ind _{\cT (\rM \times \rM)}$ below)
from the $K_0$-group of the tangent bundle to
$$K_0 (C^* (\rM \times \rM)) \cong K_0 (\cK),$$
where $\cK$ denotes the $C^*$-algebra of compact operators. %added
Applying a mapping cone arguments as described in, for example, \cite[Section 3]{Rouse;BlupInd} and \cite[Section 4.1]{So;Ktheory},
one proves that the index map corresponding to \eqref{Exact} is equal to $\ind _{\cT (\rM \times \rM)}$ composed with the principal symbol.
Next, Connes constructed a topological index map, by embedding $\rM $ into $\bbR^N$ for some large enough $N$,
and considering Thom isomorphism and Morita equivalence (between groupoids).
Then he showed that these two index maps coincide.

Connes' proof of the Atiyah-Singer index theorem triggered a large number of subsequent works in index theory through Lie groupoids.
%Index theory for a regular foliation $\mathcal F \subset T M$ was
%exploited by Connes and Skandalis \cite{Connes;FoliationInd,ConnesSkandalis;FoliationInd}.
%For more details, one may consult \cite{MooreSchochet}.
For instance, based on the analysis and index problems on manifolds with corners, many groupoids were constructed 
\cite{Rouse;BlupInd, Rouse;CornerFredholm, Monthubert;BoundaryInd, Debord;FiberedCorners, Monthubert;BdK, Nistor;TopIndCorn, Nistor;SingularSpaces}.
These constructions and index theorems depend heavily on the existence of boundary defining functions and embedding the manifold under question into a cube instead of $\bbR ^N$.
Meanwhile, Androulidakis and Skandalis associate the holonomy groupoid to a singular foliation and investigated corresponding properties \cite{Skandalis;HoloGpd},
which opens the door to explore singular foliations via Lie groupoids.
Alternatively, one can study the simplified index problem on Lie groupoids by pairing with cohomology 
\cite{Monthubert;BdCoh, Connes;FoliationInd, Tang;Ind, Tang;GpdAction, Pflaum;GpoidLocalizedInd}.

To study the Fredholm index of fully elliptic operators on manifolds with boundary,
Carrillo-Rouse, Lescure, and Monthubert replace the tangent bundle in the adiabatic groupoid construction 
by some ``non-commutative tangent bundle'' \cite{Monthubert;BdCoh}.
In that case, it is (the $C^*$-algebra of) a sub-groupoid of the adiabatic groupoid of 
the $b$-stretched product groupoid associated to a manifold $\rM$ with boundary $\rM _1$, explicitly:
$$ (\rM_1 \times \rM_1 \times \bbR) \times (0, 1) \sqcup T \rM \times \{ 0 \} \rightrightarrows \rM \times [0, 1] $$
(it is a continuous family groupoid).
Similar arguments are utilized by Debord, Lescure, and Nistor for the case of conical pseudo-manifolds \cite{Nistor;ConicalInd}.

In this paper, we consider index formulae for pseudo-differential operators on Lie groupoids $\cH \tto M$ with two orbits,
which is isomorphic as an abstract groupoid (i.e., a set with groupoid operations) to
\begin{equation}\label{BoundaryGPD}
M_0 \times M_0 \sqcup G \times M_1 \times M_1 \rightrightarrows M_0 \sqcup M_1 = M,
\end{equation}
where $M_0 = M \setminus M_1$ is an open dense subset, 
and $G \times M_1 \times M_1 \rightrightarrows M_1$ is the product of the pair groupoid 
$M_1 \times M_1$ and the Lie group $G$ as a groupoid over a single point.
We shall further assume that the isotropy group $G$ is an exponential Lie group,
of dimension equal to the codimension of the manifold $M_1$ in $M$.
\begin{dfn}
In this paper, we shall call these Lie groupoids \emph{boundary groupoid with two leaves and exponential isotropy}.
Throughout the paper, we shall use $\cH \tto M$ to denote such a groupoid.
\end{dfn}

Note that a groupoid of the form (\ref{BoundaryGPD})
may carry different Lie groupoid structures, and we shall regard them as different objects
(see Proposition \ref{Existence} and Example \ref{Counter} below). 
These Lie groupoids are the simplest example of boundary groupoids \cite{So;FullCal},
and are often holonomy groupoids integrating some singular foliations
\cite{Nistor;Riem,Nistor;LieMfld,Skandalis;SingFoliation1,Skandalis;SingFoliation2,Debord;IntAlgebroid}.
In this case, it is not clear whether an embedding analogous to the manifold with corners exists.
Fortunately, the $K$-theory of these Lie groupoid $C^*$-algebras is computed in \cite{So;Ktheory}. 
Namely for {\em all} boundary groupoids with two leaves and exponential isotropy $\cH \tto M$, we have 
\begin{align*}
K _0 (C^* (\cH)) \cong & \bbZ, & K _1 (C^* (\cH)) \cong & \bbZ && \text{if $M_1$ is of odd codimension $\geq 3$,} \\
K _0 (C^* (\cH)) \cong & \bbZ \oplus \bbZ, & K _1 (C^* (\cH)) \cong & \{ 0 \} && \text{if $M_1$ is of even codimension,}
\end{align*}
regardless of Lie structures (see \eqref{Comp2} below). 
Hence in order to derive an index formula for elliptic operators (or just their principal symbols), 
it suffices to produce one integer in the odd codimension case and two integers in the even codimension case,
which would completely describe its index in $K_0 (C^* (\cH))$.
Moreover, the pushforward induced by the {\em extension map}
$$ K _0 (C^* (M_0 \times M_0)) \xrightarrow{\varepsilon _{\cH , M_0} } K _0 (C^* (\cH)) $$
is an isomorphism in the odd case and is injective in the even case.
This implies the Fredholm index of a fully elliptic operator on $\cH$ is also determined by its $K_0 (C^* (\cH))$ index.

In the special case when a renormalized trace can be defined, one can derive an index formula \cite{QS;RenormInd}
using renormalization techniques similar to that of \cite{Nistor;Hom2}.
It was found that in the odd case with codimension $\geq 3$, the $\eta$-term of the renormalized index formula vanishes,
hence both the Fredholm and $K_0 (C^* (\cH))$ index is just given by the Atiyah-Singer integral.
Moreover, one could expect a deeper description of the relationship between the
isomorphism $K(C^*(M \times M)) \cong K(C^*(\cH))$ and the vanishing of eta term,
that we shall prove in this paper.

%\smallskip

In this paper, we take a different approach. 
Denote by $\cG \tto M$ a {\em general} Lie groupoid with unit space $M$.
We shall introduce the notion of a {\em deformation from the pair groupoid} (See Definition \ref{deformation}),
where we deform the pair groupoid $M \times M$ to our desired groupoid $\cG\tto M$.
Here, we would like to point out two major differences between our definition and
that for the tangent groupoid and the adiabatic groupoid, even for the construction of the deformation to the normal cone
\cite{Debord;AdiabaticGpd, Debord;GpdAction,Debord;Blup}:
\begin{enumerate}
\item For the tangent groupoid, the fiber at $t=0$ is the tangent bundle of $M$; whereas,
for a deformation from the pair groupoid, the fiber at $t=0$ is the groupoid $\cG$, usually {\em not} a vector bundle.
\item Given a closed manifold $M$, there always is the associated tangent groupoid. But a deformation from the pair groupoid for $\cG\tto M$ may not exist.
(Thus, we exploit a little obstruction theory for the existence of a deformation from the pair groupoid.)
\end{enumerate}
Hence, unlike the tangent groupoid and adiabatic groupoid, our idea seems to go backwards.
That is, given a Lie groupoid $\cG \tto M$,
we ask if $\cG\tto M$ can be obtained by deforming the pair groupoid $M\times M$,
in other words, we are looking for a bigger Lie groupoid $\cD$ which realizes such deformation.

For a Lie groupoid $\cG\tto M$, once a deformation $\cD$ from the pair groupoid $M \times M$ exists,
we shall construct a deformation index map
$$\ind _\cD : K_0 (C^* (\cG)) \to K_0 (C^*(M \times M \rightrightarrows \rM)) \cong \bbZ.$$
which is the key ingredient to establish the following theorem.

\begin{thm}[Theorem \ref{Main1Pf}]\label{Main1}
Let $\cG\tto M$ be a Lie groupoid and $A$ the associated Lie algebroid of $\cG$.
Suppose that a deformation $\cD$ from the pair groupoid $M \times M$ exists. Then one has the commutative diagram
\begin{equation}
\xymatrix@C=7em{
K_0 (C^* (\rA))\ar[r]^{\ind _{\cT (\cG)}} \ar[d]^{\cong} & K_0 (C^* (\cG)) \ar[d]^{\ind _\cD} \\
K_0 (C^* (T \rM ))\ar[r]^{\ind _{\cT (\rM \times \rM)}} & K_0 (C^* (\rM \times \rM)),
}
\end{equation}
where the top map is just the analytic index map constructed via the adiabatic groupoid and the bottom map is the Atiyah-Singer index.
\end{thm}

The above theorem can be used to simplify index problems on Lie groupoids to those on the pair groupoid of the unit space. 
In particular, for a boundary groupoid with two orbits and exponential istropy, %i.e., of the form (\ref{BoundaryGPD}), 
we have the following theorem, which identifies index maps on such groupoids.

\begin{thm}[Theorem \ref{Main2Pf}]\label{Main2}
Suppose that $\cH \tto M$ is any boundary groupoid with two orbits and exponential istropy, 
and a deformation from the pair groupoid exists for $\cH$. Then the composition
$$ K _0 (C^* (M_0 \times M_0)) \cong \bbZ \xrightarrow{\varepsilon _{\cH, M_0} }
K _0 (C^* (\cH)) \xrightarrow{\ind _{\cD}} K_0 (C^* (M \times M)) \cong \bbZ $$
is an isomorphism.
\end{thm}

As applications, we apply the pairing considered in \cite{Pflaum;GpoidLocalizedInd} to obtain our index formulae
on boundary groupoids with two leaves and exponential isotropy.

\begin{thm}[Theorems \ref{OddCase} and \ref{EvenCase}]
Suppose that $\cH \tto M$ is any boundary groupoid with two orbits and exponential istropy, 
and a deformation from the pair groupoid exists for $\cH$.
Let $\varPsi$ be an elliptic pseudo-differential operator on $\cH$.
One has the index formulae:
\begin{enumerate}
\item if $\dim G \geq 3$ is odd, then
$$
\ind (\varPsi) = \ind _{\cT (M \times M)} (\partial [\sigma (\varPsi) ])
= \int _{T^* M} \big\langle \hat A (T ^* M) \wedge \ch (\sigma [\varPsi ]) , \Omega _{\pi^! T M} \big\rangle;
$$
\item if $\dim G $ is even, then 
\begin{align*}
\ind (\varPsi)
=& \int _{T^* M} \big\langle \hat A (T ^* M) \wedge \ch (\sigma [\varPsi ]) , \Omega _{\pi^! T M} \big\rangle \\ \nonumber
& \bigoplus \int _{T^* M_1 \times \mathfrak g} \big\langle \hat A (T ^* M_1 \oplus \mathfrak g)
\wedge \ch (\sigma [\varPsi ] \vert _{M_1}) , \Omega' _{\pi^! T M_1} \big\rangle,
\end{align*}
where $\mathfrak{g}$ is the Lie algebra of $G$.
\end{enumerate}
\end{thm}

Again, it is interesting to observe that the above index formulae do not depend on the Lie structure of the boundary groupoid in question.

\subsection{Structure of the paper} Section 2 is devoted to reviewing basic definitions and facts related to Lie groupoids,
such as Lie algebroids, boundary groupoids, submanifold groupoids, the tangent groupoid.
In Section \ref{deformation gpd}, given a Lie groupoid $\cG \tto M$,
we define the notion of {\em a deformation from the pair groupoid} to $\cG$,
show that such deformation exists for submanifold groupoids,
and briefly discuss obstruction to existence of such deformation groupoids.
Then, under the assumption that a deformation $\cD$ from the pair groupoid to $\cG$ exists,
we construct explicitly an index map
$$\ind _\cD : K_0 (C^* (\cG)) \to K_0 (C^*(\rM \times\rM)) \cong \bbZ,$$
which is similar to that of the tangent (or adiabatic) groupoid,
and show that it is compatible with the analytic index map, hence implies Theorem \ref{Main1}.
Finally in Section \ref{FredholmFully}, after reviewing the Atyah-Singer formula appearing in \cite{Pflaum;GpoidLocalizedInd},
we combine these results with Theorem \ref{Main1} to establish index formulae
for (fully) elliptic (pseudo)-differential operators on boundary groupoids with two orbits and exponential isotropy $\cH \tto M$. 
These index formulae are essentially the Atiyah-Singer formula,
hence we give a $K$-theoretic proof of the results in \cite{QS;RenormInd} without using a renormalized trace.

\smallskip
\textbf{Acknowledgment.}
We would like to thank Prof. Xiang Tang for explaining his paper \cite{Pflaum;GpoidLocalizedInd} to us.
We also thank the anonymous referee for a careful reading of the manuscript and many suggestions that have improved it substantially.
Qiao was partially supported by the NSFC (grant nos. 11971282, 12271323).

\medskip
\section{Preliminaries}\label{basics}

\subsection{Lie groupoids, Lie algebroids, Pseudodifferential operators on Lie groupoids, and groupoid $C^*$-algebras}
We first review some basic knowledge of Lie groupoids, Lie algebroids, pseudodifferential calculus
on Lie groupoids, and groupoid $C^*$-algebras \cite{Nistor;LieMfld, Nistor;Family, Nistor;GeomOp, MacBook87, MacBook05, MonPie, NWX;GroupoidPdO, RenBook80, So;FullCal}.

\begin{dfn}
A Lie groupoid $\cG \tto M$ consists of the following data:
\begin{enumerate}
\item[(i)]
 manifolds $M$, called the space of units, and $\cG$;
\item[(ii)]
 a unit map (inclusion) $\bu: M \rightarrow G$;
\item[(iii)]
submersions $\bs,\bt: \cG \rightarrow M$, called the source and target maps respectively, satisfying
$$\bs \circ \bu= \id_M =\bt\circ \bu ;$$
\item[(iv)]
 a multiplication map $\bm: \cG^{(2)}:=\{(g, h)\in \cG \times \cG \, : \, \bs(g) = \bt(h)\} \rightarrow \cG$, $(g,h) \mapsto gh$, which
is associative and satisfies
$$\bs(gh)=\bs(h), \quad \bt(gh)=\bt(g), \quad g(\bu \circ \bs (g))=g =(\bu\circ \bt(g))g;$$
\item[(v)]
an inverse diffeomorphism $\bi: \cG \rightarrow \cG$, $g \mapsto g=g^{-1}$, such that $\bs(g^{-1})=\bt(g)$, $\bt(g^{-1})= \bs(g)$, and
$$g g^{-1}=\bu(\bt(g)), \quad g^{-1}g= \bu(\bs(g)).$$
\end{enumerate}
All maps above are assumed to be smooth.
\end{dfn}

%\begin{remark}
%%In general, the space $\cG_1$ is not required to be Hausdorff. However, since $\bs$ is a submersion,
%%it follows that each fiber $\cG_x:=\bs^{-1}(x)$ is a smooth manifold, hence it is Hausdorff.
%In this paper, all groupoids are assumed to be Hausdorff, and all maps are assumed to be smooth.
%\end{remark}

\begin{dfn}
A {\em homomorphism} between Lie groupoids $\cG\tto M_1$ and $\cH\tto M_2$ is by definition 
a functor $\phi: \cG \rightarrow \cH$ which is smooth both on the unit space $M_1$ and on $\cG$.
Two Lie groupoids $\cG$ and $\cH$ are said to be {\em isomorphic} if there are homomorphisms
$\phi: \cG \rightarrow \cH$ and $\psi: \cH \rightarrow \cG$ such that $\psi\circ \phi$ and 
$\phi\circ \psi$ are identity homomorphisms on $\cG$ and $\cH$ respectively. %\green{Made this a definition}
\end{dfn}

Lie groupoids are closely related with Lie algebroids. Here we recall the definition.

\begin{dfn}
A \em{Lie algebroid} $A$ over a manifold $M$ is a vector bundle $A$
over $M$, together with a Lie algebra structure on the space
$\Gamma ^\infty (A)$ of the smooth sections of $A$ and a bundle map $\nu: A \rightarrow TM$, 
called the \emph{anchor} map, such that 
%extended to a map between sections of theses bundles, such that
$$ [X,f Y]=f[X,Y]+(\nu(X)f)Y,$$
for all smooth sections $X$ and $Y$ of $A$ and any smooth function $f$ on $M$. 
\end{dfn}

Given a Lie groupoid $\mathcal{G}$ with units $M$,
we can associate a Lie algebroid $A(\cG)$ to $\cG$ as follows. (For more details, see \cite{MacBook87, MacBook05}.)
The $\bs$-vertical subbundle of $T\mathcal{G}$ for $\bs:\mathcal{G}\rightarrow M$ is denoted by $T^\bs(\mathcal{G})$ and called simply the $\bs$-vertical
bundle for $\mathcal{G}$. It is an involutive distribution on
$\mathcal{G}$ whose leaves are the components of the $\bs$-fibers of
$\mathcal{G}$. (Here involutive distribution means that
$T^\bs(\mathcal{G})$ is closed under the Lie bracket, i.e. if $X,Y \in
\mathfrak{X}(\mathcal{G})$ are sections of $T^\bs(\mathcal{G})$, then
the vector field $[X,Y]$ is also a section of $T^\bs(\mathcal{G})$.) Hence
we obtain
\begin{equation*}
  T^\bs\mathcal{G}= \Ker \bs_*=\displaystyle{\bigcup_{x\in       M}T\mathcal{G}_x}\subset T\mathcal{G}.
\end{equation*}
The {\em Lie algebroid} of $\cG$, denoted by $A(\mathcal{G})$ (or simply $A$ sometimes), is
defined to be $T^\bs(\mathcal{G})\big\vert_M$, the restriction of the
$\bs$-vertical tangent bundle to the set of units $M$. 
In this case, we say that $\cG$ {\em integrates} $A(\cG)$.

\begin{rem}
Given a Lie algebroid $\rA$, there may not exist a Lie groupoid integrating $\rA$.
When $\rA$ is almost regular (i.e., $\nu$ has constant rank on an open dense subset of $\rM$), 
Debord's quasi-graphoid construction can be used to integrate $\rA$ \cite{Debord;IntAlgebroid}.
If in particular $\rA$ is of the form as in (iv) of Definition \ref{Def} with $\mathfrak g _k$ solvable, 
then the resulting groupoid (which is a quasi-graphoid) is necessarily of the form
$$ \cG = (\rM_0 \times \rM _0) \sqcup \big( \sqcup _i \rM_i \times \rM_i \times \rG \big), $$
as (ii) of Definition \ref{Def}.
However Debord's quasi-graphoid is not always Hausdorff. 
Fortunately, the result of \cite{So;Ktheory} follows from the composition series (see \ref{Comp} below) and Thom isomorphism,
which only requires $\rG$ to be {\em locally Hausdorff} and $\rM$ being Hausdorff.

If, furthermore, all leaves of $\rA$ are simply connected, 
then Nistor's gluing construction \cite{Nistor;IntAlg'oid} integrates $\rA$ to a Hausdorff Lie groupoid $\cG$ with simply connected $\bs$-fibers.
Moreover uniqueness implies $\cG$ is again necessarily of the form as (ii) of Definition \ref{Def}.

Since we shall only consider almost regular Lie algebroids, 
we can always integrate them into Lie groupoids, where the $K$-theoretic calculation can be applied.
\end{rem}

\begin{exam}
The Lie algebroid of the pair groupoid $M \times M$ is the tangent bundle $TM$ with the usual Lie bracket on vector fields and the anchor map is the identity.
\end{exam}

%Two Lie groupoids may share the same Lie algebroid.
%For instance, both the pair groupoid $M \times M$ and the fundamental groupoid associated to $M$ integrate the tangent Lie algebroid $TM$.
%
%\smallskip

%\subsection{Groupoid pseudo-differential operators}
Let $E \to M$ be a vector bundle.
Recall \cite{NWX;GroupoidPdO} that an $m$-th order pseudo-differential operator on $\cG$ is a right invariant,
smooth family $P = \{ P _x \} _{x \in M}$,
where each $P _x $ is an $m$-th order classical pseudo-differential operator on sections of $\bt ^* \rE \to \bs ^{-1} (x)$.
We denote by $\Psi^m (\cG, E)$ (resp. $\rD ^m (\cG , E)$) the algebra of uniformly supported,
order $m$ classical pseudo-differential operators (resp. differential operators).

Recall \cite{Nistor;Family} that one defines the strong norm for $P \in \Psi^0(\cG, E)$
$$ \| P \| := \sup _\rho \| \rho (P) \|, $$
where $\rho $ ranges over all bounded $*$-representations of $\Psi ^0 (\cG, E)$ satisfying
$$ \| \rho (P) \| \leq \sup _{x \in M}
\Big\{ \int _{\bs ^{-1} (x)} | \kappa _P (g) |\, \mu _x , \int _{\bs ^{-1} (x)} | \kappa _P (g ^{-1}) | \, \mu _x \Big\},$$
whenever $P \in \Psi ^{- \dim M - 1} (\cG, E)$ with (continuous) kernel $\kappa _P$.

\begin{dfn}
The $C^*$-algebras $\fU (\cG)$ and $C ^* (\cG)$ are defined to be the completion of $\Psi ^0 (\cG, E)$ and $\Psi ^{- \infty} (\cG, E)$ respectively
with respect to the strong norm $\| \cdot \| $.
\end{dfn}

One also defines the reduced $C^*$-algebras $\fU _r (\cG)$ and $C^* _r (\cG)$ by completing
$\Psi ^0 (\cG, E)$ and $\Psi ^{- \infty} (\cG, E)$) respectively with respect to the reduced norm
$$ \| P \| _r := \sup _{x \in M} \big\{ \| P _x \| _{L^2 (\bs ^{-1} (x))} \big\}.$$
Recall that if the strong and reduced norm coincide, then $\cG$ is called {\em (metrically) amenable},
which is the case for the groupoids we shall consider.

%\begin{exam}
%Let $M$ be a smooth manifold. Then the tangent bundle $TM$ has a Lie groupoid structure,
%by regarding its fibers as a bundle of Lie groups. 
%Its reduced groupoid $C^*$-algebra $C_r^*(TM)$ can be canonically identified
%with $C_0(T^*M)$ via Fourier transform, where $T^*M$ is the cotangent bundle of $M$.
%This example can be generalized to vector bundles.
%\end{exam}

\smallskip
\subsection{Invariant submanifolds and composition series}
Let $\cG \tto M$ be a Lie groupoid.
\begin{dfn}
Let $S$ be any subset of $M$.
We denote by
$$\cG_S:=\bs ^{-1} (S) \cap \bt ^{-1} (S)$$
the {\em reduction} of $\cG$ to $S$.
The reduction $\cG _S$ is a sub-groupoid of $\cG$.
In particular, if $S=\{x\}$, then $\cG_x := \cG_S$ is called the {\em isotropy group} at $x$.

If $S \subseteq M$ is an embedded submanifold such that $ \bs^{-1} (S) = \bt ^{-1} (S) $,
we say that $S \subset M$ is an {\em invariant submanifold}.
\end{dfn}

\begin{dfn}\label{restrictionmap}
Given a closed invariant submanifold $S $ of $\cG$.
For any groupoid pseudo-differential operator $P = \{ P _x \} _{x \in M} \in \Psi ^m (\cG, E)$, we define the restriction
$$ P |_{\cG _{S}} (P) := \{ P _x \} _{x \in \cG _{S}} \in \Psi ^m (\cG _{S}, E).$$
Restriction extends to a map from $\fU (\cG) $ to $\fU (\cG _S)$ and also from $C^* (\cG)$ to $C^* (\cG _S)$.
We denote both such restriction maps, and also the induced $K$-group homomorphisms, by $r _{\cG , \rS}$.
\end{dfn}

\begin{nota}\label{notations}
Let $\cG \rightrightarrows M$ be a Lie groupoid, $\rU$ be an open subset of $\rM$.
Then $\cG _\rU := \bs ^{-1} (\rU) \cap \bt ^{-1} (\rU) \rightrightarrows \rU$ is an open sub-groupoid of $\cG$.
Any element in $C^* (\cG _\rU)$ extends to $C^* (\cG )$ by $0$.
We denote such extension map by $\varepsilon _{\cG , \rU}$.
It is a homomorphism of $C^*$ algebras and hence induces a map from $K_\bullet (C^* (\cG _\rU))$ to $K_\bullet (C^* (\cG))$,
which we shall still denote by $\varepsilon _{\cG , \rU}$.

\end{nota}

%\smallskip

%\subsection{Elliptic and Fredholm operators}
Now suppose we are given a groupoid $\cG \rightrightarrows M$ ($M $ not necessarily compact),
and invariant submanifolds $M_0 , M_1, \cdots M_r$,
such that their closures $\bar M_i$ are also invariant submanifolds that furthermore satisfies
$$ M = \bar{M _0} \supset \bar{M _1} \supset \cdots \supset \bar{ M _r }$$
(such setting is natural for boundary groupoids in Definition \ref{Def} below).
For simplicity we shall denote $\bar{\cG_i} := \cG_{\bar{M _i}}$.
Denote by $S A'$ the sphere sub-bundle of the dual of the Lie algebroid $A(\cG)$ of $\cG$. 

\begin{dfn}\label{FullyElliptic}
Let $\sigma : \Psi ^m (\cG) \to C^\infty (S A')$ denotes the principal symbol map.
For each $i = 1, \cdots , r$, define the {\em joint symbol maps}
\begin{equation}
\label{JointSym}
\bj _i : \Psi ^m (\cG) \to C^\infty (S A') \oplus \Psi ^m (\bar \cG _i),
\quad \bj _i (P ) := (\sigma (P ) , P |_{\bar \cG _i}).
\end{equation}
The map $\bj _i$ extends to a homomorphism from $\fU (\cG)$ to $C_0 (S A ^*) \oplus \fU (\bar \cG _i)$.

We say that $P \in \Psi ^m (\cG)$ is {\em elliptic} if $\sigma (P)$ is invertible,
and it is called {\em fully elliptic} if $\bj _1 (P )$ is invertible (which implies $P$ is elliptic).
\end{dfn}

\begin{dfn}
Denote by $\cJ _0 := \overline {\Psi ^{-1} (\cG)} \subset \fU (\cG)$,
and $\cJ _i \subset \fU (\cG), i = 1, \cdots , r$ the null space of $\bj _{r - i + 1}$.
\end{dfn}

By construction, it is clear that
$$ \cJ _0 \supset \cJ _1 \supset \cdots \supset \cJ _r .$$
Also, any uniformly supported kernels in $\Psi ^{- \infty} (\cG _{\bar M _i \setminus \bar M _j}, E |_{\bar M _i \setminus \bar M _j})$,
can be extended to a kernel in $\Psi ^{- \infty }  (\cG _{\bar M _i}, E |_{\bar M _i })$ by zero.
This induces a $*$-algebra homomorphism from $C^* (\cG _{\bar M _i \setminus \bar M _j} ) $ to $ C^* (\cG _{\bar M _i}) $.
We shall use the following key fact.

\begin{lem}\cite[Lemma 2 and Theorem 3]{Nistor;Family}
\label{Comp}
One has short exact sequences
\begin{align}
& 0 \to \cJ _{i+1} \to \cJ _i \to C ^* (\cG _{\bar M _i \setminus \bar M _{i +1}}) \to 0 , \\
& 0 \to C^* (\cG _{\bar M _i \setminus \bar M _j} ) \to C^* (\bar \cG _i) \to C^* (\bar \cG _j) \to 0, \quad \Forall j > i.
\end{align}
\end{lem}

%\smallskip
\subsection{Boundary groupoids and submanifold groupoids}
In this paper we are interested in some more specific classes of groupoids.
To begin with, let us recall the definition of boundary groupoids in \cite{So;FullCal,So;PhD}.

\begin{dfn}\label{Def}
Let $\cG \rightrightarrows M$ be a Lie groupoid with $M$ compact.
We say that $\cG$ is a boundary groupoid if:
\begin{enumerate}
\item[(i)]
the singular foliation defined by the anchor map $\nu : A \to TM $ has finite number of leaves
$M _0 , M _1, \cdots , M _r \subset M$
(which are invariant submanifolds),
such that $\dim M = \dim M _0 > \dim M _1 > \cdots > \dim M _r $;
\item[(ii)]
For all $k = 0, 1, \cdots , r$, $\bar M _k := M _k \cup \cdots \cup M _r $ are closed submanifolds of $M$;
\item[(iii)]
For $k = 0$, $\cG _0 := \cG _{M _0}$ is the pair groupoid,
and for $k = 1, 2, \cdots , r$, we have $\cG _k := \cG _{M _k} \cong G _k \times M _k \times M _k $ for some Lie group $G _k$;
\item[(iv)]
For each $k = 0, 1, \cdots , r$, there exists an unique sub-bundle $\bar A _k \subset A |_{\bar M _k}$ such that
$\bar A _k |_{M _k} = \ker (\nu |_{M _k}) $ ($= \mathfrak g _k \times M _k$).
\end{enumerate}
\end{dfn}

Boundary groupoids are closely related to Fredholm groupoids and blowup groupoids.
Roughly speaking, Fredholm groupoids are those on which Fredholmness of a pseudodifferential operator is completely characterized by its
ellipticity and invertibility at the boundary.
For the definition and basic properties of Fredholm groupoids, one may consult \cite{Nistor;Fredholm1, Nistor;Fredholm2,Come19}.
The basic result relevant to our discussion is that boundary groupoids are often amenable and Fredholm groupoids.

\begin{lem}
(See \cite[Lemma 7]{Nistor;GeomOp})
For any boundary groupoid of the form $\cG = (M _0 \times M _0) \cup (\bbR ^q \times M _1 \times M _1)$,
$$ C^* (\cG) \cong C^* _r (\cG).$$
In other words, $\cG$ is metrically amenable.
\end{lem}

Moreover, since the additive group $\bbR^q$ is amenable, the product groupoid $\bbR ^q \times M _1 \times M _1$ is topologically amenable.
By \cite[Theorem 4.3]{Come19}, we have the following proposition.

\begin{proposition}
The boundary groupoid $\cG = (M _0 \times M _0) \cup (\bbR ^q \times M _1 \times M _1)$ is a Fredholm groupoid.
\end{proposition}

Given a boundary groupoid, one naturally considers the sequence of invariant submanifolds
$$ M \supset \bar M _1 \supset \cdots \supset \bar M _r, $$
where $\bar M _i$ is given to be (ii) of Definition \ref{Def}.
We have the short exact sequence
\begin{align}
& 0 \to C^* (\cG _{\bar M _i \setminus \bar M _j} ) \to C^* (\bar \cG _i) \to C^* (\bar \cG _j) \to 0, \quad \Forall j > i,
\end{align}
which induces the $K$-theory six-terms exact sequences:
\begin{equation}\label{Comp2}
\begin{CD}
K _1 (C^* (\cG _{M _i} )) @>>> K _1 (C^* (\bar \cG _i)) @>>> K _1(C^* (\bar \cG _{i+1})) \\
@AAA @. @VVV \\
K _0(C^* (\bar \cG _{i+1})) @<<< K _0 (C^* (\bar \cG _i)) @<<< K _0 (C^* (\cG _{M _i }))
\end{CD}
\end{equation}
If $r = 1$, there is only one exact sequence in \eqref{Comp2}. 
Moreover if $G _1$ is solvable, connected and simply connected (i.e. exponential),
then Connes' Thom isomorphism readily gives the $K$-groups of (the $C^*$-algebra of) $\bar \cG _1 = G \times M_1 \times M_1$.
Using these facts, Carrillo-Rouse and the second author computed the $K$-theory of boundary groupoids 
with two leaves and exponential isotropy, namely:
\begin{lem}
Suppose that a boundary groupoid $\cG$ is of the form $\cG = (M _0 \times M _0) \cup (G \times M _1 \times M _1)$ with $G$ an exponential Lie group
(i.e., $\cG$ is a boundary groupoid with two leaves and exponential isotropy).
One has
\begin{align*}
K _0 (C^*(\cG)) \cong \bbZ, & \quad K _1 (C^*(\cG)) \cong \bbZ, & \text{if $\dim G \geq 3$ odd}; \\
K _0 (C^*(\cG)) \cong \{0\}, & \quad K _1 (C^*(\cG)) \cong \bbZ \oplus \bbZ, & \text{if $\dim G$ even.}
\end{align*}
\end{lem}

Next, we recall {\em submanifold groupoids}, which form a sub-class of boundary groupoids in \cite{So;Ktheory}.

\begin{exam}[Submanifold groupoids]\label{SubGPD}
Suppose that $M _1 $ is a closed embedded sub-manifold of $M$ of codimension $q \geq 2$.
Let $f \in C^\infty (M)$ be any smooth function which vanishes on $M_1$ and strictly positive over $M_0=\rM \setminus M_1$
(if $M$ is connected then $M_0 = M \setminus M_1$ is still connected and $f$ must either be strictly positive or strictly negative on $M_0$).
Then there exists a Lie algebroid structure over
$$\rA_f:= T \rM \to \rM$$
with anchor map
$$(z,w)\mapsto (z,f(z)\cdot w) , \quad \Forall (z,w)\in T \rM,$$
and Lie bracket on sections
$$ [X, Y] _{A _f} := f [X, Y] + (X \cdot f) Y - (Y \cdot f ) X .$$
This Lie algebroid is almost injective
(since the anchor is injective, in fact an isomorphism, over the open dense subset $\rM _0 = M\setminus \rM_1)$,
and hence it integrates to a Lie groupoid that is a quasi-graphoid (see \cite[Theorems 2 and 3]{Debord;IntAlgebroid}).

By \cite{Nistor;IntAlg'oid}, $\rA _f$ integrates to a Lie groupoid of the form
$$ \cH = \rM _0 \times \rM _0 \sqcup \bbR ^q \times \rM _1 \times \rM _1 \rightrightarrows M=M_0 \sqcup M_1, $$
a boundary groupoid with two leaves and exponential isotropy.
This groupoid is called a submanifold groupoid.
\end{exam}

\begin{exam}[Renormalizable boundary groupoids]\label{RenormGPD}
For even more specific examples of submanifold groupoids, in \cite{QS;RenormInd}, we considered
$$f = r ^N $$
for some fixed even integers $N$, where $r := d (M_1 , \cdot)$ is the Riemannian distance function from $M_1$ with respect to some metric.
The resulting groupoid is called a {\em renormalizable boundary groupoid}.
In \cite{QS;RenormInd}, it was shown that the $\eta$-term for renormalizable groupoids vanishes in the index formula
via the method of renormalized trace.
\end{exam}

%\smallskip
\subsection{The tangent groupoid and the adiabatic groupoid }
%To understand the map $\ind _\cD$ above, recall that for the pair groupoid $\rM \times\rM \rightrightarrows \rM$,
Let us recall that for a closed manifold $M$,
Connes \cite{Connes;Book} constructed the tangent groupoid, which is a Lie groupoid of the form
$$ \cT (\rM \times \rM) := \rM \times \rM \times (0, 1] \sqcup T \rM \times \{ 0 \} \rightrightarrows \rM \times [0, 1] ,$$
where the differentiable structure of $\cT (\rM \times \rM)$ is defined by fixing some Riemannian metric on $\rM$ and then using the following maps
from some open subsets of $T \rM \times [0, 1] $ to $ \cT (\rM \times \rM) $
\begin{eqnarray*}
&&\mathbf x (p, X, \epsilon) := (p, \exp _p (- \epsilon X), \epsilon), \text{ for } \epsilon > 0, \\
&&\mathbf x (p, X, 0)  := (p, X, 0),
\end{eqnarray*}
as charts.

The tangent groupoid construction generalizes to any Lie groupoid $\cG$.
One considers the adiabatic groupoid
$$ \cT (\cG) := \cG \times (0, 1] \sqcup \rA \times \{ 0 \} \rightrightarrows \rM \times [0, 1],$$
where $\rA$ is the Lie algebroid of $\cG$.
Debord and Skandalis \cite{Debord;Blup} further generalized the construction
by replacing the set of units $\rM \subset \cG$ with arbitrary sub-groupoid $\cH \subset \cG$,
and considered gluing $\mathcal N ^\cG _\cH $, the normal bundle of $\cH$ in $\cG$, to $\cG \times \bbR \setminus \{ 0 \} $.
The resulting object is naturally a Lie groupoid, which they call the deformation to the normal cone.

For simplicity we return to the case of the adiabatic groupoid $\cT (\cG)$. 
One naturally constructs the index map
\begin{equation}
\label{AnalyticInd}
\ind _{\cT (\cG)} := r_{\cT (\cG), \rM \times \{1 \}} \circ r_{\cT (\cG), \rM \times \{0 \}} ^{-1},
\end{equation}
where $r_{\cT (\cG), \rM \times \{1 \}} : C^* (\cT (\cG)) \to C^* (\cG) $ and
$ r_{\cT (\cG), \rM \times \{0 \}} : C^* (\cT (\cG)) \to C^* (\rA) $ are respectively the restriction maps to the sub-groupoids of $\cT (\cG))$ over
$ \rM \times \{ 1 \}$ and $\rM \times \{ 0 \}$.

Let $\varPsi$ be any classical elliptic pseudo-differential operator on $\cG \tto M$.
Its principal symbol $\sigma (\varPsi)$ is an invertible element in $C (S \rA')$,
where $A'$ is the dual bundle of $A$ and $SA'$ denotes the sphere bundle of $A'$.
We identify $C (A')$, the algebra of continuous functions on $A'$ vanishing at infinity with $C^* (A)$, 
the (fiber-wise) convolution $C^*$-algebra, through Fourier transform to obtain
$$  K_0 (C (A')) \cong K_0 (C^* (A)),$$
and let
$$\partial : K _1 (S \rA') \to K_0 (C (\rA')) \cong K_0 (C^* (\rA)) $$
be the connecting map induced by the short exact sequence
$$ 0 \to C (\rA') \to C (\bar \rA') \to (S \rA') \to 0 $$
(where $\bar \rA'$ denotes the bundle of solid balls in $\rA'$).
Then by \cite{So;Ktheory}, the analytic index of $\varPsi $ is given by
$$ \ind _{\cT (\cG)} (\partial [\sigma (\varPsi)]) \in C^* (\cG),$$
which, if $\varPsi $ is of order 0, coincides with the image of $[\varPsi] \in K_1 (\fU (\cG))$ under the connecting map induced by the short exact sequence
$$0 \to C^* (\cG) \to \fU (\cG) \to \fU / C^* (\cG) \to 0.$$

\medskip
\section{The deformation groupoid and the deformation index map}\label{deformation gpd}

Motivated by the tangent groupoid construction in the previous section, we introduce the following definition, which plays a central role in the paper.

\begin{dfn}\label{deformation}
Let $\cG \rightrightarrows \rM$ be a Lie groupoid, not necessarily a boundary groupoid.
{\em A deformation from the pair groupoid $M \times M$ to $\cG$} is a Lie groupoid $\cD \rightrightarrows \rM \times [0, 1]$
that is isomorphic as an abstract groupoid to
$$ \rM \times \rM \times (0, 1] \sqcup \cG \rightrightarrows \rM \times (0, 1] \sqcup \rM $$
(where $\rM \times \rM \times (0, 1] \rightrightarrows \rM \times (0, 1]$ is the product of the pair groupoid and the set $(0, 1]$).
\end{dfn}

\begin{example}
Let $M$ be a closed manifold. Regard $TM\rightarrow M$ as a  Lie groupoid, then Connes' tangent groupoid,
$\cT (\rM \times \rM)$ gives a deformation from the pair groupoid to $TM$.
\end{example}

%For convenience, we define two maps as follows.
%\begin{dfn}
%Define the restriction maps
%\begin{align*}
%r _{\cD , \rM \times \{0\}} &: K_\bullet (C^* (\cD)) \to K_\bullet (C^*(\cD \vert_{\rM \times \{ 0 \}})) \cong K_\bullet (C^* (\cG)), \\
%\quad  r _{\cD , \rM \times \{1\}} &: K_\bullet (C^* (\cD)) \to K_\bullet (C^*(\cD \vert_{\rM \times \{ 0 \}})) \cong K_\bullet (C^* (\rM \times \rM)) .
%\end{align*}
%\end{dfn}

%In this paper, we confine ourselves to the case that $\cG$ is a boundary groupoid.

\subsection{Construction of deformations from the pair groupoid}
The following fact is another motivation to introduce the above definition.

\begin{proposition}\label{Existence}
There exists a deformation from the pair groupoid for submanifold groupoids defined in Example \ref{SubGPD}.
\end{proposition}
\begin{proof}
We use the same notation as in Example \ref{SubGPD}.
To construct the groupoid $\cD \rightrightarrows \rM \times [0, 1] $,
we begin with defining the Lie algebroid of $\cD$.
We set $\tilde A _f \to \rM \times [0, 1]$ to be the vector bundle pullback of $\rA _f $ ($\cong T M$).
Sections of $\tilde \rA _f$ are just vector fields on $\rM \times [0, 1]$ which are tangential to the $\rM$ direction.
Hence one can define the anchor map
$$\nu((z, s),w) = ((z, s),(f(z) + s)\cdot w) , \quad \Forall (z,w)\in T \rM,  s\in[0,1], $$
and Lie bracket on sections by
$$ [X, Y] _{\tilde A _f} := (f + s) [X, Y] + (X \cdot f) Y - (Y \cdot f ) X .$$
When restricted to $M \times \{ 0 \}$,
the Lie algebroid $\tilde \rA _f \big \vert_{M \times \{ 0 \}}$ is isomorphic to $\rA _f $. 
On $\rM \times \{ s \}$ for each $s > 0$, since $f+s$ is positive and bounded away from $0$, $\nu$ defines an isomorphism
$$ \tilde \rA _f \big \vert_{M \times \{ s \}} \cong T M.$$
Using \cite{Debord;IntAlgebroid} and \cite[Theorem 1]{Nistor;IntAlg'oid}, 
one observes that $\tilde \rA _f $ integrates to a Lie groupoid of the form
$$ \cD = \cH \sqcup (\rM \times \rM \times (0, 1]) \rightrightarrows \rM \times [0, 1],$$
which is what we need.
\end{proof}

%\smallskip
\subsection{Obstruction to the existence of deformations from the pair groupoid}
In this subsection we point out some notable necessary condition for $\cG$ in order for a deformation from the pair groupoid to exist.

\begin{proposition}\label{LieAlgIso}
Suppose that $\cG \tto \rM$ is a Lie groupoid such that a deformation from the pair groupoid exists.
Denote by $\rA$ the Lie algebroid of $\cG$. Then we have 
$$ \rA \cong T \rM,$$
as vector bundles.
\end{proposition}
\begin{proof}
Let $\cD \tto \rM \times [0, 1]$ be a deformation from the pair groupoid. Denote its Lie algebroid by $\tilde \rA \to \rM \times [0, 1]$.
Then one has
$$ \tilde \rA \big \vert _{\rM \times \{ 0 \}} \cong \rA , \quad \tilde \rA \big \vert _{\rM \times \{ 1 \}} \cong T \rM .$$
Hence an isomorphism can be constructed by, say, 
fixing a connection $\tilde \nabla$ on $\tilde \rA$ and then parallel transport along the family of curves $(p, t), t \in [0, 1]$ for each $p \in \rM$.
\end{proof}

The existence of a deformation from the pair groupoid also imposes necessary conditions 
on the structural vector fields of $\cG$ (i.e. the image of the anchor map). 
For simplicity, we consider the case when $\rM_1$ is a single point $\{ p \}$.
Suppose a deformation $\cD$ from the pair groupoid for $\cG$ exists.
Denote the Lie algebroids of $\cG$ and $\cD$ by $\rA$ and $\tilde \rA$ respectively, and their anchor maps by $\nu$ and respectively $\tilde \nu$.
Then it is obvious that
\begin{align*}
\tilde \rA \big \vert _{\rM \times (0, 1]} \cong & \, T \rM \times (0, 1], \\
\tilde \rA \big \vert _{\rM \times \{ 0 \}} \cong & \, \rA .
\end{align*}
Suppose that $s$ is a nowhere vanishing local section of $\rA$ around $p$, such that $\nu (s) (p) = 0$.
Then the section $s$ extends to some nowhere vanishing local section $\tilde s$ of $\tilde \rA$ over an open neighborhood of $\{p \} \times [0, 1]$.
Because $\tilde \nu $ is injective except at $(p, 0)$,
it follows that $\tilde \nu (\tilde s)$ is a family of local vector fields on $\rM$ parameterized by $[0, 1]$,
which is non-vanishing on $\rM \times \{ t\}$, $1 \geq t > 0$ and has an isolated zero at $p$ on $\rM \times \{ 0 \}$.
This implies that the Hopf-Poincar\'e index of $\nu (s) = \tilde \nu (\tilde s) \big \vert _{\rM \times \{ 0 \}}$ equals zero.
Hence, we are able to construct the following counterexample which shows that a boundary groupoid
may not possess a deformation from the pair groupoid.

%\begin{exam}
%There exist Lie algebroids spanned by structural vector fields of non-zero degrees.
%For example, on $\mathbb R^2$, the vector fields
%$$ x \tfrac{\partial}{\partial x} + y \tfrac{\partial}{\partial y} \,\, \text{ and } \,\, y \tfrac{\partial}{\partial x} - x \tfrac{\partial}{\partial y} $$
%have isolated zero of degree 1, and commute each other.
%Moreover, their flows give an action of the abelian Lie group $\bbR^2$ on itself.
%Thus the tangent Lie algebroid of the action groupoid $\bbR ^2 \rtimes \bbR^2 $ is spanned by these vector fields,
%which contradicts to the above necessary condition that the Hopf-Poincar\'e index is equal to zero.
%\end{exam}

\begin{exam}
\label{Counter}
Let $\bbS ^2 = \{ (x, y, z) \in \bbR^3 : x^2 + y^2 +z^2 \} $ be the 2-sphere.
As a set, we identify 
$$ \bbS ^2 = \{ (0, 0, -1) \} \sqcup \bbR^2 $$ 
by stereographic projection.
The vector fields
$$ X := \tfrac{\partial }{\partial u} , Y := \tfrac{\partial}{\partial v} $$
on $\bbR ^2$ extends smoothly to $\bbS ^2$ by $0$.
Indeed, using stereographic coordinates on $\bbS ^2 \setminus \{ (0, 0, 1) \}$, one has
$$ X = -((u')^2 - (v')^2) \tfrac{\partial }{\partial u'} 
- 2 u' v' \tfrac{\partial }{\partial v'} $$
and similar expression for $Y$.
These vector fields commute and integrates to an action of $\bbR ^2$ on $\bbS ^2$,
and one obtains the action groupoid $\cG = \bbS ^2 \rtimes \bbR ^2$.
Thus the structural vector fields of the tangent Lie algebroid of $\cG$ is spanned by $\{X, Y \}$,
which contradicts the above necessary condition that the Hopf-Poincar\'e index is equal to zero.

Lastly, note that the action groupoid $\bbS ^2 \rtimes \bbR ^2$ (by right translation) is isomorphic to the pair groupoid $\bbR^2 \times \bbR^2$,
therefore $\cG$ is a boundary groupoid.
\end{exam}

\begin{remark}
Example \ref{Counter} exhibits that there are certain obstructions to the existence of such deformation for boundary groupoids.
At present, besides Proposition \ref{LieAlgIso}, 
we do not know of any other necessary or sufficient conditions that guarantee the existence of such deformations in general, 
which will be left as future work.
\end{remark}

\medskip
\subsection{The deformation index map}\label{IndexMap}

In this subsection, we always assume that a deformation from the pair groupoid
$$ \cD := \cG \sqcup (\rM \times \rM \times (0, 1]) \rightrightarrows \rM \times [0, 1]$$
exists for the given Lie groupoid $\cG \tto M$.

We shall use $\cD$ to construct an index, similar to \cite{Connes;Book, Zenobi;SecondaryInv}.
Observe that $\cD$ has closed saturated subgroupoids $\cD \vert_{\rM \times \{ 0 \}} \cong \cG , \cD \vert _{\rM \times \{1\}} \cong \rM \times \rM$. 
Because the sequence
\begin{align*}
K_\bullet (C^*(\cD \vert_{\rM \times ( 0, 1] }))
& \cong K_\bullet (C^* (\rM \times \rM \times (0, 1])) \cong 0
\xrightarrow{\varepsilon _{\cD , \rM \times (0, 1]}} K _\bullet (C^* (\cD)) \\
& \xrightarrow{r _{\cD , \rM \times \{0\}}} K _\bullet (C^* (\cG))
\xrightarrow{\partial} K_{\bullet + 1} (C^*(\cD \vert_{\rM \times ( 0, 1]})) \cong 0
\end{align*}
is exact, therefore the map $r _{\cD , \rM \times \{0\}}$ is invertible,
where the extension map $\varepsilon _{\cD , \rM \times (0, 1]}$ is defined in Notation \ref{notations},
and the restriction map $r _{\cD , \rM \times \{0\}}$ is defined in Definition \ref{restrictionmap}.

Hence it is natural to introduce the following index map.

\begin{dfn}
The deformation index map is defined to be
$$ \ind _{\cD} := r _{\cD , \rM \times \{1\}} \circ r _{\cD , \rM \times \{0\}} ^{-1} : K_\bullet (C^* (\cG)) \to K_\bullet (C^* (\rM \times \rM)).$$
\end{dfn}

Using this deformation index map $\ind_{\cD}$, we establish the following `index comparison' theorem,
which is one of the main results in the paper.

\begin{thm}\label{Main1Pf}
Let $\cG\tto M$ be a Lie groupoid. Suppose there exists a deformation groupoid $\cD$ from the pair groupoid $M\times M$.
Then one has the commutative diagram
\begin{equation}
\xymatrix@C=7em{
K_0 (C^* (\rA))\ar[r]^{\ind _{\cT (\cG)}} \ar[d]^{\cong} & K_0 (C^* (\cG)) \ar[d]^{\ind _\cD} \\
K_0 (C^* (T \rM ))\ar[r]^{\ind _{\cT (\rM \times \rM)}} & K_0 (C^* (\rM \times \rM)),
}
\end{equation}
where the top map is just the analytic index map constructed via the adiabatic groupoid and the bottom map is the Atiyah-Singer index.
\end{thm}

\begin{proof}
Let $\cT (\cD) \rightrightarrows (\rM \times [0, 1]) \times [0, 1]$ be the adiabatic groupoid of $\cD$.
Recall that as a set $\cT (\cD) = \cD \times (0, 1] \sqcup \tilde A $.
Hence $\cT (\cD)$ restricted to $(\rM \times \{ 0 \}) \times [0, 1]$ and $(\rM \times \{ 1 \}) \times [0, 1]$ are respectively
\begin{align*}
\cT (\cG) \rightrightarrows \rM \times [0, 1] =& \cG \times (0, 1] \sqcup A, \\
\text{ and } \cT (\rM \times \rM) \rightrightarrows \rM \times [0, 1] =& \rM \times \rM \times (0, 1] \sqcup T \rM,
\end{align*}
the adiabatic groupoid of $\cG$ (respectively the adiabatic groupoid of $\rM \times \rM$).
One can further restrict $\cT (\cG)$ and $\cT (\rM \times \rM)$ to $\rM \times \{ 0 \}$ or $\rM \times \{ 1 \}$, resulting in the commutative diagram
\begin{equation}
\xymatrix@C=7em{
C^* (\rA) & C^* (\cT (\cG)) \ar[l]_{r_{\cT (\cG), \rM \times \{0 \}}} \ar[r]^{r_{\cT (\cG), \rM \times \{1 \}}} & C^* (\cG) \\
C^* (\tilde \rA ) \ar[u]_{r_{\tilde \rA , \rM \times \{0\}}} \ar[d]^{r_{\tilde \rA , \rM \times \{1\}}}
& C^*(\cT (\cD)) \ar[l] _{r_{\cT (\cD), (\rM \times [0,1]) \times \{0\}}}  \ar[r]^{r_{\cT (\cD), (\rM \times [0,1]) \times \{1\}}}
\ar[u]_{r_{\cT (\cD), (\rM \times \{0\}) \times [0,1]}} \ar[d]^{r_{\cT (\cD), (\rM \times \{1\}) \times [0,1]}}
& C^* (\cD) \ar[u]_{r_{\cD, \rM \times \{0 \}}} \ar[d]^{r_{\cD, \rM \times \{1 \}}} \\
C^* (T\rM) & C^* (\cT (\rM \times \rM)) \ar[l]^{r_{\cT (\rM \times \rM), \rM \times \{0 \}}} \ar[r]_{r_{\cT (\rM \times \rM), \rM \times \{1 \}}}
& C^* (\rM \times \rM).
}
\end{equation}
Then we pass to the corresponding $K$-group maps.
Observe that $r_{\cT (\cD), (\rM \times \{0\}) \times [0, 1]}$ fits in the six-term exact sequence
\begin{align*}
K_\bullet (C^* (\cT (\cD) _{ (\rM \times (0, 1]) \times [0, 1]}))
&\xrightarrow {\varepsilon _{\cT (\cD), (\rM \times (0, 1]) \times [0, 1]}}
K_\bullet (C^* (\cT (\cD))) \\
&\xrightarrow {r_{\cT (\cD), (\rM \times \{0\}) \times [0, 1]}}
K_\bullet (C^* (\cT (\cG))) \to \cdots,
\end{align*}
and moreover we have
$$ \cT (\cD) _{ (\rM \times (0, 1]) \times [0, 1]} \cong \cT (\cD _{\rM \times (0, 1]}) \cong \cT (\rM \times \rM \times (0, 1])
\cong \cT (\rM \times \rM ) \times (0, 1], $$
whose convolution $C^*$-algebra is contractible.
Therefore the map $r_{\cT (\cD), (\rM \times \{0\}) \times [0, 1]}$ is invertible.
Similarly, the maps $r_{\cT (\cG), \rM \times \{0 \}}$, $r_{\tilde \rA , \rM \times \{0\}}$, $r_{\cT (\cD), (\rM \times [0,1]) \times \{0\}}$,
$r_{\cT (\rM \times \rM), \rM \times \{0 \}}$ and  $r_{\cD, \rM \times \{0 \}}$ are all isomorphisms between corresponding $K$-groups.

Recall \eqref{AnalyticInd} that
\begin{align*}
\ind _{\cT (\cG)} :=& r_{\cT (\cG), \rM \times \{1 \}} \circ r_{\cT (\cG), \rM \times \{0 \}} ^{-1}, \\
\ind _{\cT (\rM \times \rM)} :=& r_{\cT (\rM \times \rM), \rM \times \{1 \}} \circ r_{\cT (\rM \times \rM), \rM \times \{0 \}} ^{-1}
\end{align*}
are just Connes' analytic index maps.
Lastly, observe from the proof of Proposition \ref{LieAlgIso} that $\tilde A \cong A \times [0, 1]$.
Hence for any $u \in C^* (A)$ that is furthermore a Schwartz function,
let $\tilde u \in C^* (\tilde A)$ be the pullback of $u$
(i.e., $\tilde u $ is just $u$ on each $A \times \{s\}$).
Then 
$$u \mapsto \tilde u \big \vert _{M \times \{ 1 \}} $$
induces the $K$-theory map
$$ r_{\tilde \rA , \rM \times \{1\}} \circ r_{\tilde \rA , \rM \times \{0\}} ^{-1} $$
(which is equivalent to identifying $A \cong T M$).
Hence one ends up with the commutative diagram
\begin{equation}
\xymatrix@C=7em{
K_0 (C^* (\rA))\ar[r]^{\ind _{\cT (\cG)}} \ar[d]^{\cong} & K_0 (C^* (\cG)) \ar[d]^{\ind _\cD} \\
K_0 (C^* (T \rM ))\ar[r]^{\ind _{\cT (\rM \times \rM)}} & K_0 (C^* (\rM \times \rM)),
}
\end{equation}
which completes the proof.
\end{proof}

We are in position to prove the following theorem for boundary groupoids with two orbits and exponential isotropy,
which implies Theorem \ref{Main2}. 
Recall that the restriction map $r$ and the extension map $\varepsilon$ are defined in
Definition \ref{restrictionmap} and Notation \ref{notations}, respectively.
%\red{NEED to check the notations!!!}

\begin{thm}\label{Main2Pf}
Let $\cH= M_0 \times M_0 \sqcup G \times M_1 \times M_1 \rightrightarrows M = M_0\sqcup M_1$ be a boundary groupoid
with two orbits and exponential isotropy.
Suppose further that there exists a deformation $\cD$ from the pair groupoid $M \times M$ for $\cH$.
Then one has
$$ \ind _\cD \circ \varepsilon _{\cH , M_0} = \varepsilon _{M \times M, M_0}.$$
\end{thm}
\begin{proof}
Observe that
$$ M_0 \times M_0 \times [0, 1] = \cD _{M_0 \times [0, 1]}$$
is an open subgroupoid.
For any $u \in C^* (M_0 \times M_0)$,
let $\tilde u \in C^* (M_0 \times M_0 \times [0, 1])$ be the pullback of $u$ by the projection to the $M_0 \times M_0$ factor
(i.e., $\tilde u $ is just $u$ on each $M_0 \times M_0 \times \{s\}$).
Extend $\tilde u $ to $\cD$ by 0 and one gets $\varepsilon _{\cD , M_0 \times [0, 1]} (\tilde u)$.
When restricted to $s=0$ and $s=1$, it is clear that $\varepsilon _{\cD , M_0 \times [0, 1]} (\tilde u)$ equals
$$ \varepsilon _{\cH , M_0 } (u) \,\,\text{ and } \,\,\varepsilon_{M \times M, M_0} (u), $$
respectively.
Passing to $K$-theory, one obtains
$$ r _{\cD, M \times \{0\} } \circ \varepsilon _{\cD , M_0 \times [0, 1]} ([\tilde u]) = \varepsilon _{\cH , M_0 } ([u]),$$
and
$$r _{\cD, M \times \{1\} } \circ \varepsilon _{\cD , M_0 \times [0, 1]} ([\tilde u]) = \varepsilon_{M \times M, M_0} ([u]), $$
for any class $[u] \in K_\bullet (M_0 \times M_0)$.
Hence, we have
\begin{align*}
\ind _\cD \circ \varepsilon _{\cH , M_0 } ([u])
=& r _{\cD, M \times \{1\} } \circ r_{\cD, M \times \{0\} } ^{-1} \circ r _{\cD, M \times \{0\} }
\circ \varepsilon _{\cD , M_0 \times [0, 1]} ([\tilde u]) \\
=& \varepsilon_{M \times M, M_0} ([u]), \qedhere
\end{align*}
which completes the proof.
\end{proof}

%We turn to the proof of Theorem \ref{Main1}, which is reformulated and proved as follows.
%\red{It seems that we need to get rid of the symbol ``$f$'' in the proof because $f$ is related to submanifold groupoids!!!}

\medskip
\section{Fredholm and $K$-theoretic index of (fully) elliptic operators on boundary groupoids}\label{FredholmFully}

%In this section, we shall compute the index by defining certain maps into the pair groupoid case, 
%whose index can in turn be computed using the Atiyah-Singer index formula.
{In this section, we combine the map in Theorem \ref{Main1} and the Atiyah-Singer index formula to compute the index.
Before doing that, we briefly recall the Atiyah-Singer index formula,
in particular, we shall use the version appearing in \cite[Theorem 5.1]{Pflaum;GpoidLocalizedInd}.

Let $\cG \tto M$ be a Lie groupoid and $A$ the associated Lie algebroid of $\cG$.
Define the line bundle
$$ L := \wedge ^{top} T^* M \otimes \wedge ^{top} A .$$
We suppose further that $\cG$ is unimodular, i.e., there exists an invariant nowhere vanishing section $\Omega $ of $L$.
Then the characteristic map $\chi _\Omega $ defines a map from the groupoid cohomology to cyclic cohomology by \cite[Equation (5)]{Pflaum;GpoidLocalizedInd}
\begin{eqnarray*}
 &&\chi _\Omega (\varphi _1 \otimes \cdots \otimes \varphi _k)(a_0 \otimes \cdots \otimes a_k)\\
&:=& \int_\rM \Big( \int _{g_0 \cdots g_k = 1_x} a_0 (g_0) \varphi _1 (g_1) a_1 (g_1) \cdots \varphi _k (g_k) a_k (g_k) \Big) \Omega (x)
\end{eqnarray*}
(on the level of cochains), for any groupoid cocycle $\varphi _1 \otimes \cdots \otimes \varphi _k$.
Then one combines $\chi _\Omega$ with the canonical pairing between cyclic homology and cohomology,
the Connes-Chern character map, which maps from $K_0 (\Psi _c ^{-\infty} (\cG))$ to the cyclic homology of $\Psi _c ^{-\infty} (\cG)$,
and the van Est map,
to obtain a pairing
$$\langle [\psi ], \alpha \rangle _\Omega \in \bbC $$
for any $[\psi ] \in K_0 (\Psi _c ^{-\infty} (\cG)), \alpha \in H^\bullet (\rA , \bbC)$.

Then the main result of \cite[Theorem 5.1]{Pflaum;GpoidLocalizedInd} states that such pairing can be computed by the following formula.

\begin{lem}
\label{PPTMain}
For any elliptic pseudo-differential operator $\varPsi$ and $ \alpha \in H^{2 k} (\rA , \bbC)$, we have
$$ \langle \ind (\varPsi), \alpha \rangle _\Omega
= (2 \pi \sqrt {-1})^{-k} \int _{\rA'} \big \langle \pi^* \alpha \wedge \hat A (\rA') \wedge \ch (\sigma [\varPsi ]) , \Omega _{\pi^! \rA} \big \rangle ,$$
where $\hat A (\rA') \in H ^\bullet (T \rM)$ is the $A$-hat genus, 
$\ch : K _0 (C (\rA')) \to H ^{even} (T \rM)$ is the Chern Character (which can be defined by, say, the Chern-Weil construction),
and $\pi^! A$ is the pull-back Lie algebroid of $A$ along the projection $\pi: A' \rightarrow M$.
\end{lem}

Let us consider the particular case when $\cG = \rM \times \rM$ is the pair groupoid, hence $A = T \rM$, and $L$ is trivial.
One obtains an invariant nowhere vanishing section $\Omega $ of $L$ by considering
\begin{equation}
\label{Section}
\Omega := C d x^1 \wedge \cdots \wedge d x^{\dim \rM}
\otimes \tfrac {\partial }{\partial x^1} \wedge \cdots \wedge \tfrac {\partial }{\partial x^{\dim \rM}}
\end{equation}
locally, for any constant $C \neq 0$.
Also, the constant function $1$ obviously defines a class $[1] \in H^0 (\rA , \bbC)$.
Applying Lemma \ref{PPTMain},
one gets
\begin{equation}
\label{PPTCor1}
\langle \ind (\varPsi), 1 \rangle _\Omega
= \int _{T^* \rM} \langle \hat A (T ^* \rM) \wedge \ch (\sigma [\varPsi ]) , \Omega _{\pi^! T \rM} \rangle .
\end{equation}
Since $\ind (\varPsi)$ is an integer, by choosing an appropriate normalization $C$ for $\Omega$, we have
$$ \ind (\varPsi) = \langle \ind (\varPsi), 1 \rangle _\Omega, $$
and the identity \eqref{PPTCor1} simplifies to
\begin{equation}
\label{PPTCor2}
\ind (\varPsi)
= \int _{T^* \rM} \langle \hat A (T ^* \rM) \wedge \ch (\sigma [\varPsi ]) , \Omega _{\pi^! T \rM} \rangle.
\end{equation}

\begin{rem}
In order to apply Lemma \ref{PPTMain}, one needs a non-trivial $\cG$-invariant Lie algebroid class.
By \cite{Lu;Modular} the construction of such class is non-trivial,
because the Lie algebroid associated to a boundary groupoid is in general not unimodular.
However, in the following subsections we shall use a deformation from the pair groupoid to simplify the problem to that of a regular groupoid,
which is clearly unimodular,
so that the Lemma \ref{PPTMain} can be applied.
\end{rem}

In the sequel, we focus on boundary groupoids with two orbits and exponential isotropy, i.e., of the form
\begin{equation}\label{GpdForm}
\cH := M_0 \times M_0 \sqcup G \times M_1 \times M_1 \rightrightarrows M = M_0\sqcup M_1,
\end{equation}
where $G$ is an exponential Lie group, and assume that a deformation from the pair groupoid exists for $\cH$.

%\smallskip
\subsection{The odd co-dimension case}
Recall that if the isotropy subgroup $G$ is exponential and of odd dimension $\geq 3$, then
\begin{equation}
K _0 (M_0 \times M_0 ) \cong \bbZ \xrightarrow{\varepsilon _{\cH , M_0}} C^* (\cH) \cong \bbZ
\end{equation}
is an isomorphism.
Identifying $K _0 (M_0 \times M_0 ) \cong K _0 (M \times M)$ using $\varepsilon _{M \times M, M_0}$,
one sees from Theorem \ref{Main2} that $\ind _\cD $ is the inverse of $\varepsilon _{\cH , M_0}$.

Given any elliptic pseudodifferential operator $\varPsi$,
in order to compute the integer 
$$\ind (\varPsi) = \ind _{\cT (\cH)} (\partial [\sigma (\varPsi)]),$$
we apply Theorem \ref{Main1} to obtain
\begin{equation}\label{AS1}
\ind (\varPsi) = \ind _{\cT (\cH)} (\partial [\sigma (\varPsi)]) = \ind _{\cT (M \times M)} (\partial [\sigma (\varPsi)]),
\end{equation}
where we regard $\partial [\sigma ] \in K_0 (C^* (T M))$ on the rightmost expression.
Observe that $\ind _{\cT (M \times M)} (\partial [\sigma ])$ is just the analytic index of the pair groupoid,
Hence we apply Lemma \ref{PPTMain} (with the normalization of \eqref{PPTCor2}) to conclude  the following theorem.

\begin{thm}\label{OddCase}
Suppose that $\cH \tto M $ is a boundary groupoid with two orbits and exponential isotropy with $\dim G \geqslant 3$ odd,
and a deformation from the pair groupoid exists for $\cH$.
Let $\varPsi$ be an elliptic pseudo-differential operator on $\cH$.
Then one has the index formula
\begin{equation}\label{AS}
\ind (\varPsi) = \ind _{\cT (M \times M)} (\partial [\sigma (\varPsi) ])
= \int _{T^* M} \langle \hat A (T ^* M) \wedge \ch (\sigma [\varPsi ]) , \Omega _{\pi^! T M} \rangle.
\end{equation}
\end{thm}

\begin{remark}
By Proposition 3.3, we see that a deformation from the pair groupoid always exists for submanifold groupoids.
Because renormalizable boundary groupoids form a special class of submanifold groupoids, 
Theorem \ref{OddCase} applies and implies that the index is given only by the Atiyah-Singer term.
Comparing with the results of \cite{Bohlen;Rev1,Bohlen;HeatIndex} (for higher co-dimension) and also \cite{QS;RenormInd},
one sees that the $\eta$-term vanishes for elliptic pseudo-differential operators on $\cG$.
Therefore we give a completely new proof that generalizes these previous results.
\end{remark}

\begin{exam}
Recall \cite[Section 6]{Nistor;GeomOp} that for any Lie groupoid the vertical de Rham operator $d$ is an (invariant) groupoid differential operator.
Its vector representation computes its Lie algebroid cohomology.
Fixing a Riemannian metric on $\rA$, one can then define the formal adjoint $d^*$ of $d$, which leads one to construct the Euler operator $d + d^*$,
and also the signature operator.
More generally, one can construct generalized Dirac operators.
 
For illustration here, we only consider the case when the Lie algebroid $\rA$ is spin,
and $\cH$ is of the form \eqref{GpdForm} with $M_1$ of odd codimension, where a deformation from the pair groupoid exists.
Let $S_\rA$ be the spinor bundle of $\rA$.
Following \cite[Section 6]{Nistor;GeomOp}, one constructs the (groupoid) Dirac operator $D _\rA$ associated with the Levi-Civita connection.
One the other hand, Lemma \ref{LieAlgIso} implies that $M$ is also a spin manifold, 
and the spinor bundle $S _{T M}$ is isomorphic to $S _\rA$.
Hence it is standard to construct the (groupoid) Dirac operator $D$ associated with the Levi-Civita connection.
One sees from its explicit construction that the principal symbol of $D _A$ is equal to that of $D$
(however $D_\rA$ and $D$ are very different as groupoid differential operators).
Hence by Theorem \ref{Main1Pf}, the index of $D _\rA$ is the same as the index of $D$.
Theorem \ref{OddCase} then gives an explicit formula for $\ind (D _\rA)$.
Lastly recall \cite[Section 3]{BGV} that one can generalize the construction of the Dirac operator by tensoring  $S_\rA$ with a vector bundle $W$,
and in this case the Chern form $\ch (\sigma [D _\rA])$ can be written explicitly using the twisting curvature.
\end{exam}

\smallskip
\subsection{The even co-dimension case}
On the other hand, if $\rG$ is of even dimension, one has short exact sequence
\begin{equation}
0 \to K _0 (M_0 \times M_0 ) \cong \bbZ \xrightarrow{\varepsilon _{\cH , M_0}} K_0 (C^* (\cH))
\xrightarrow{r _{\cH, M_1}} K_0 (C^* (\cH_{M_1})) \cong \bbZ  \to 0
\end{equation}
and Theorem \ref{Main2} canonically identifies
$$ K_0 (C^* (\cH)) \cong \bbZ \oplus \bbZ $$
via $\ind _\cD \oplus r _{\cH, M_1} $.
For the $\ind _\cD$ component, the arguments for the odd case apply without any change,
and the result is the same Atiyah-Singer integral \eqref{AS}.
To compute
$$ r _{\cH, M_1} \circ \ind _{\cT (\cH)} (\partial [\sigma ]), $$
we observe that restriction of $\cT (\cH)$ to $M_1 \times [0, 1]$ is just $\cT (\cH_1)$ which is the adiabatic groupoid of $\cH_1 = \rG \times M_1 \times M_1$,
hence we obtain the commutative diagram
\begin{equation}
\xymatrix@C=7em{
C^* (\rA) \ar[d]^{r _{\rA , M_1}}
& C^* (\cT (\cH)) \ar[l]_{r_{\cT (\cH), M \times \{0 \}}} \ar[r]^{r_{\cT (\cH), M \times \{1 \}}} \ar[d]^{r _{\cT (\cH), M_1 \times [0, 1]}}
& C^* (\cH) \ar[d]^{r _{\cH, M_1}} \\
C^* (\rA \vert _{M_1}) & C^* (\cT (\cH_1)) \ar[l]^{r_{\cT (\cH_1), M_1 \times \{0 \}}} \ar[r]_{r_{\cT (\cH_1), M \times \{1 \}}} & C^* (\cH_1),
}
\end{equation}
which implies
$$ r _{\cH, M_1} \circ \ind _{\cT (\cH)} (\partial [\sigma ]) = \ind _{\cT (\cH_1)} (\partial [\sigma] \vert _{M_1}). $$
The right hand side of the above can, in turn be computed by Lemma \ref{PPTMain} with $\cH_1 = M_1 \times M_1 \times \rG$ as the groupoid:
\begin{equation}
\ind _{\cT (\cH_1)} (\partial [\sigma] \vert _{M_1})
= \int _{T^* M_1 \times \mathfrak g} \big \langle \hat A (T ^* M _1 \oplus \mathfrak g)
\wedge \ch (\sigma [\varPsi ] \vert _{M_1}) , \Omega' _{\pi^! T M_1} \rangle ,
\end{equation}
where $\Omega ' \in \Gamma ^\infty (\wedge ^{top} \mathfrak g \otimes \wedge ^{top} T^* M_1 \otimes \wedge ^{top} T M_1)$
is a suitably normalized, invariant nowhere vanishing section defined in the same manner as Equation \eqref{Section},
and $\mathfrak{g}$ is the Lie algebra of the isotropy group $G$.

To conclude, we have arrived at an index formula for elliptic pseudodifferential operators on $\cH$.

\begin{thm}\label{EvenCase}
Suppose that $\cH \tto M$ is a boundary groupoid with two orbits and exponential isotropy,
i.e., of the form (\ref{GpdForm}) with $\dim G $ even,
and a deformation from the pair groupoid exists for $\cH$.
Let  $\varPsi$ be an elliptic pseudo-differential operator on $\cH$.
One has the index formula
\begin{align*}
\ind (\varPsi)
=& \int _{T^* M} \big\langle \hat A (T ^* M) \wedge \ch (\sigma [\varPsi ]) , \Omega _{\pi^! T M} \big \rangle \\ \nonumber
& \bigoplus \int _{T^* M_1 \times \mathfrak g} \big \langle \hat A (T ^* M_1 \oplus \mathfrak g)
\wedge \ch (\sigma [\varPsi ] \vert _{M_1}) , \Omega' _{\pi^! T M_1} \big \rangle \in \bbZ \oplus \bbZ \cong K_0 (C^* (\cH)).
\end{align*}
\end{thm}

%\smallskip
\subsection{The Fredholm index for fully elliptic operators}
Lastly, recall the definition of fully elliptic operators in Definition \ref{FullyElliptic}. 
We suppose in this subsection that $\varPsi$ is a fully elliptic operator on $\cH \tto M$, where $\cH$ is a boundary groupoid with two orbits and exponential isotropy.

\begin{corollary}
The Fredholm index of $\varPsi$ is
$$ \ind _F (\varPsi ) = \int _{T^* M} \big \langle \hat A (T ^* M) \wedge \ch (\sigma [\varPsi ]) , \Omega _{\pi^! T M} \big \rangle .$$
\end{corollary}
\begin{proof}
Recall that $\varPsi$ is invertible modulo $C^* (M _0 \times M_0) \cong \cK$ and its Fredholm index lies in $K_0 (C^* (M_0 \times M_0)) $.
By Theorems \ref{Main2}, \ref{Main1} and \cite{So;Ktheory} we have
\begin{align*}
\ind _F (\varPsi ) =& \ind _\cD \circ \varepsilon _{\cH, M_0} (\ind _F (\varPsi)) \\
=& \ind _\cD (\ind _{\cT (\cH)} (\sigma (\varPsi ))) \\
=& \int _{T^* M} \big \langle \hat A (T ^* M) \wedge \ch (\sigma [\varPsi ]) , \Omega _{\pi^! T M} \big \rangle. \qedhere
\end{align*}
\end{proof}

\begin{remark}
One can replace the pair groupoid in Definition \ref{deformation} by other groupoids. 
Most of the arguments in Section 4 still works in a more general setting. 
It would be interesting to see which class of groupoids can arise from this kind of (non-trivial) deformations, 
and what index formulae one can obtain.
\end{remark}

\bigskip

%\bibliographystyle{plain}
%\bibliography{Reference}

\begin{thebibliography}{10}

\bibitem{Rouse;BlupInd}
I.~Akrour and P.~Carrillo~Rouse.
\newblock Longitudinal b-operators, blups and index theorems.
\newblock arXiv:1711.11197v3, 2018.

\bibitem{Nistor;Riem}
B.~Ammann, R.~Lauter, and V.~Nistor.
\newblock On the geometry of {R}iemannian manifolds with a {L}ie structure at
  infinity.
\newblock {\em Int. J. Math. Math. Sci.}, (1-4):161--193, 2004.

\bibitem{Nistor;LieMfld}
B.~Ammann, R.~Lauter, and V.~Nistor.
\newblock Pseudodifferential operators on manifolds with a {L}ie structure at
  infinity.
\newblock {\em Ann. of Math. (2)}, 165(3):717--747, 2007.

\bibitem{Skandalis;HoloGpd}
I.~Androulidakis and G.~Skandalis.
\newblock The holonomy groupoid of a singular foliation.
\newblock {\em J. Reine Angew. Math.}, 626:1--37, 2009.

\bibitem{Skandalis;SingFoliation1}
I.~Androulidakis and G.~Skandalis.
\newblock The analytic index of elliptic pseudodifferential operators on a
  singular foliation.
\newblock {\em J. K-Theory}, 8(3):363--385, 2011.

\bibitem{Skandalis;SingFoliation2}
I.~Androulidakis and G.~Skandalis.
\newblock Pseudodifferential calculus on a singular foliation.
\newblock {\em J. Noncommut. Geom.}, 5(1):125--152, 2011.

\bibitem{BGV}
N.~Berline, E.~Getzler, and M.~Vergne.
\newblock {\em Heat kernels and {D}irac operators}.
\newblock Grundlehren Text Editions. Springer-Verlag, Berlin, 2004.
\newblock Corrected reprint of the 1992 original.

\bibitem{Bohlen;Rev1}
K.~Bohlen.
\newblock Groupoids and singular manifolds.
\newblock arXiv:1601.04166, 2016.

\bibitem{Bohlen;HeatIndex}
K.~Bohlen and E.~Schrohe.
\newblock Getzler rescaling via adiabatic deformation and a renormalized index
  formula.
\newblock {\em J. Math. Pures Appl. (9)}, 120:220--252, 2018.

\bibitem{Rouse;CornerFredholm}
P.~Carrillo~Rouse and J.-M. Lescure.
\newblock Geometric obstructions for {F}redholm boundary conditions for
  manifolds with corners.
\newblock {\em Ann. K-Theory}, 3(3):523--563, 2018.

\bibitem{Monthubert;BdCoh}
P.~Carrillo~Rouse, J.-M. Lescure, and B.~Monthubert.
\newblock A cohomological formula for the {A}tiyah-{P}atodi-{S}inger index on
  manifolds with boundary.
\newblock {\em J. Topol. Anal.}, 6(1):27--74, 2014.

\bibitem{Monthubert;BoundaryInd}
P.~Carrillo~Rouse and B.~Monthubert.
\newblock An index theorem for manifolds with boundary.
\newblock {\em C. R. Math. Acad. Sci. Paris}, 347(23-24):1393--1398, 2009.

\bibitem{So;Ktheory}
P.~Carrillo~Rouse and B.~K. So.
\newblock K-theory and index theory for some boundary groupoids.
\newblock {\em Results Math.}, 75(4):Paper No. 172, 20, 2020.

\bibitem{Nistor;Fredholm1}
C.~Carvalho, V.~Nistor, and Y.~Qiao.
\newblock Fredholm criteria for pseudodifferential operators and induced
  representations of groupoid algebras.
\newblock {\em Electron. Res. Announc. Math. Sci.}, 24:68--77, 2017.

\bibitem{Nistor;Fredholm2}
C.~Carvalho, V.~Nistor, and Y.~Qiao.
\newblock Fredholm conditions on non-compact manifolds: theory and examples.
\newblock In {\em Operator theory, operator algebras, and matrix theory},
  volume 267 of {\em Oper. Theory Adv. Appl.}, pages 79--122.
  Birkh\"{a}user/Springer, Cham, 2018.

\bibitem{Come19}
R.~C\^{o}me.
\newblock The {F}redholm property for groupoids is a local property.
\newblock {\em Results Math.}, 74(4):Paper No. 160, 33, 2019.

\bibitem{Connes;FoliationInd}
A.~Connes.
\newblock Sur la th\'{e}orie non commutative de l'int\'{e}gration.
\newblock In {\em Alg\`ebres d'op\'{e}rateurs ({S}\'{e}m., {L}es
  {P}lans-sur-{B}ex, 1978)}, volume 725 of {\em Lecture Notes in Math.}, pages
  19--143. Springer, Berlin, 1979.

\bibitem{Connes;Book}
A.~Connes.
\newblock {\em Noncommutative geometry}.
\newblock Academic Press, Inc., San Diego, CA, 1994.

\bibitem{Debord;IntAlgebroid}
C.~Debord.
\newblock Holonomy groupoids of singular foliations.
\newblock {\em J. Differential Geom.}, 58(3):467--500, 2001.

\bibitem{Nistor;ConicalInd}
C.~Debord, J.-M. Lescure, and V.~Nistor.
\newblock Groupoids and an index theorem for conical pseudo-manifolds.
\newblock {\em J. Reine Angew. Math.}, 628:1--35, 2009.

\bibitem{Debord;FiberedCorners}
C.~Debord, J.-M. Lescure, and F.~Rochon.
\newblock Pseudodifferential operators on manifolds with fibred corners.
\newblock {\em Ann. Inst. Fourier (Grenoble)}, 65(4):1799--1880, 2015.

\bibitem{Debord;AdiabaticGpd}
C.~Debord and G.~Skandalis.
\newblock Adiabatic groupoid, crossed product by {$\Bbb{R}_+^\ast$} and
  pseudodifferential calculus.
\newblock {\em Adv. Math.}, 257:66--91, 2014.

\bibitem{Debord;GpdAction}
C.~Debord and G.~Skandalis.
\newblock Pseudodifferential extensions and adiabatic deformation of smooth
  groupoid actions.
\newblock {\em Bull. Sci. Math.}, 139(7):750--776, 2015.

\bibitem{Debord;Blup}
C.~Debord and G.~Skandalis.
\newblock Blow-up constructions for {L}ie groupoids and a {B}outet de {M}onvel
  type calculus.
\newblock {\em M\"{u}nster J. Math.}, 14(1):1--40, 2021.

\bibitem{Lu;Modular}
S.~Evens, J.-H. Lu, and A.~Weinstein.
\newblock Transverse measures, the modular class and a cohomology pairing for
  {L}ie algebroids.
\newblock {\em Quart. J. Math. Oxford Ser. (2)}, 50(200):417--436, 1999.

\bibitem{Nistor;Family}
R.~Lauter, B.~Monthubert, and V.~Nistor.
\newblock Pseudodifferential analysis on continuous family groupoids.
\newblock {\em Doc. Math.}, 5:625--655, 2000.

\bibitem{Nistor;GeomOp}
R.~Lauter and V.~Nistor.
\newblock Analysis of geometric operators on open manifolds: a groupoid
  approach.
\newblock In {\em Quantization of singular symplectic quotients}, volume 198 of
  {\em Progr. Math.}, pages 181--229. Birkh\"{a}user, Basel, 2001.

\bibitem{Monthubert;BdK}
P.-Y. Le~Gall and B.~Monthubert.
\newblock {$K$}-theory of the indicial algebra of a manifold with corners.
\newblock {\em $K$-Theory}, 23(2):105--113, 2001.

\bibitem{MacBook87}
K.~Mackenzie.
\newblock {\em Lie groupoids and {L}ie algebroids in differential geometry},
  volume 124 of {\em London Mathematical Society Lecture Note Series}.
\newblock Cambridge University Press, Cambridge, 1987.

\bibitem{MacBook05}
K.~Mackenzie.
\newblock {\em General theory of {L}ie groupoids and {L}ie algebroids}, volume
  213 of {\em London Mathematical Society Lecture Note Series}.
\newblock Cambridge University Press, Cambridge, 2005.

\bibitem{M'bert;CornerGroupoids}
B.~Monthubert.
\newblock Pseudodifferential calculus on manifolds with corners and groupoids.
\newblock {\em Proc. Amer. Math. Soc.}, 127(10):2871--2881, 1999.

\bibitem{Nistor;TopIndCorn}
B.~Monthubert and V.~Nistor.
\newblock A topological index theorem for manifolds with corners.
\newblock {\em Compos. Math.}, 148(2):640--668, 2012.

\bibitem{MonPie}
B.~Monthubert and F.~Pierrot.
\newblock Indice analytique et groupo\"{\i}des de {L}ie.
\newblock {\em C. R. Acad. Sci. Paris S\'{e}r. I Math.}, 325(2):193--198, 1997.

\bibitem{Nistor;Hom2}
S.~Moroianu and V.~Nistor.
\newblock Index and homology of pseudodifferential operators on manifolds with
  boundary.
\newblock In {\em Perspectives in operator algebras and mathematical physics},
  volume~8 of {\em Theta Ser. Adv. Math.}, pages 123--148. Theta, Bucharest,
  2008.

\bibitem{Nistor;IntAlg'oid}
V.~Nistor.
\newblock Groupoids and the integration of {L}ie algebroids.
\newblock {\em J. Math. Soc. Japan}, 52(4):847--868, 2000.

\bibitem{Nistor;SingularSpaces}
V.~Nistor.
\newblock Analysis on singular spaces: {L}ie manifolds and operator algebras.
\newblock {\em J. Geom. Phys.}, 105:75--101, 2016.

\bibitem{NWX;GroupoidPdO}
V.~Nistor, A.~Weinstein, and P.~Xu.
\newblock Pseudodifferential operators on differential groupoids.
\newblock {\em Pacific J. Math.}, 189(1):117--152, 1999.

\bibitem{Tang;Ind}
M.~Pflaum, H.~Posthuma, and X.~Tang.
\newblock The index of geometric operators on {L}ie groupoids.
\newblock {\em Indag. Math. (N.S.)}, 25(5):1135--1153, 2014.

\bibitem{Tang;GpdAction}
M.~Pflaum, H.~Posthuma, and X.~Tang.
\newblock The transverse index theorem for proper cocompact actions of {L}ie
  groupoids.
\newblock {\em J. Differential Geom.}, 99(3):443--472, 2015.

\bibitem{Pflaum;GpoidLocalizedInd}
M.~J. Pflaum, H.~Posthuma, and X.~Tang.
\newblock The localized longitudinal index theorem for {L}ie groupoids and the
  van {E}st map.
\newblock {\em Adv. Math.}, 270:223--262, 2015.

\bibitem{QS;RenormInd}
Y.~Qiao and B.~K. So.
\newblock Renormalized index formulas for elliptic differential operators on
  boundary groupoids.
\newblock To appear in the J. of Operator Theory. arXiv:2108.04592.

\bibitem{RenBook80}
J.~Renault.
\newblock {\em A groupoid approach to {$C^{\ast} $}-algebras}, volume 793 of
  {\em Lecture Notes in Mathematics}.
\newblock Springer, Berlin, 1980.

\bibitem{So;PhD}
B.~K. So.
\newblock Pseudo-differential operators, heat calculus and index theory of
  groupoids satisfying the {Lauter-Nistor} condition.
\newblock PhD thesis, The University of Warwick, 2010.

\bibitem{So;FullCal}
B.~K. So.
\newblock On the full calculus of pseudo-differential operators on boundary
  groupoids with polynomial growth.
\newblock {\em Adv. Math.}, 237:1--32, 2013.

\bibitem{Zenobi;SecondaryInv}
V.~F. Zenobi.
\newblock Adiabatic groupoid and secondary invariants in {K}-theory.
\newblock {\em Adv. Math.}, 347:940--1001, 2019.

\end{thebibliography}

\def\cprime{$'$} \def\cprime{$'$}

\end{document}